\documentclass[11pt,leqno]{amsart}

\usepackage{graphicx,psfrag,amsxtra,color,url,scalerel,amsthm,xspace,mathabx,enumitem,subfig}
\usepackage{amscd}
\usepackage{multirow}
\usepackage{xy}
\xyoption{all}
\usepackage{pinlabel}
\usepackage[linktocpage]{hyperref}

\usepackage[margin=1in]{geometry}
\usepackage{setspace}
\newcommand{\C}{\mathbb{C}}
\newcommand{\R}{\mathbb{R}}

\newcommand{\Z}{\mathbb{Z}}
\newcommand{\A}{\mathcal{A}}
\newcommand{\cptwo}{\C\textup{P}^2}
\newcommand{\cpone}{\C\textup{P}^1}

\newcommand{\rptwo}{\R\textup{P}^2}

\theoremstyle{plain}
\newtheorem{theorem}{Theorem}[section]

\newtheorem{lemma}[theorem]{Lemma}
\newtheorem{prop}[theorem]{Proposition}

\theoremstyle{definition}

\newtheorem{definition}[theorem]{Definition}

\theoremstyle{remark}
\newtheorem{remark}[theorem]{Remark}
\newtheorem{example}[theorem]{Example}

\title{A new approach to the symplectic isotopy problem}
\author{Laura Starkston}

\begin{document}
	\begin{abstract}
		The symplectic isotopy conjecture states that every smooth symplectic surface in $\cptwo$ is symplectically isotopic to a complex algebraic curve. Progress began with Gromov's pseudoholomorphic curves \cite{Gr}, and progressed further culminating in Siebert and Tian's proof of the conjecture up to degree 17 \cite{ST}, but further progress has stalled. In this article we provide a new direction of attack on this problem. Using a solution to a nodal symplectic isotopy problem we guide model symplectic isotopies of smooth surfaces. This results in an equivalence between the smooth symplectic isotopy problem and an existence problem of certain embedded Lagrangian disks. This redirects study of this problem from the realm of pseudoholomorphic curves of high genus to the realm of Lagrangians and Floer theory. Because the main theorem is an equivalence going both directions, it could theoretically be used to either prove or disprove the symplectic isotopy conjecture.
	\end{abstract}
	\maketitle
	
	\section{Introduction}
	
	One of the first examples of a symplectic manifold is the complex projective plane. Despite $\cptwo$'s simplicity and homogeneous symmetries, we remain unable to answer a fundamental question about its symplectic geometry: classify its smooth symplectic submanifolds up to symplectic isotopy. The adjunction formula ensures that the genus of a smooth symplectic surface in $\cptwo$ is determined by its degree (homology class). The question of whether there is a unique smooth symplectic representative of each homology class up to symplectic isotopy is known as \emph{the symplectic isotopy problem}. Since every homology class is represented by a smooth complex algebraic curve, this is equivalent to asking whether every embedded symplectic surface in $\cptwo$ is isotopic through symplectic embeddings to a complex algebraic curve. This was solved in the affirmative by Gromov \cite{Gr} for curves of degree one and two, extended to degree three by Sikorav \cite{Sk}, and to degree six by Shevchishin \cite{Sh}. The furthest progress to date of Siebert and Tian establishes the result up to degree $17$ \cite{ST}. 
	
	The general approach which has been taken to prove the low degree symplectic isotopy classification follows Gromov's original outline. A symplectically embedded surface in $\cptwo$ is $J_0$-holomorphic with respect to some almost complex structure $J_0$ compatible with the standard symplectic structure $\omega_{std}$ on $\cptwo$. The space of almost complex structures compatible with a given symplectic structure is contractible, so there exists a path $\{J_t\}$ of almost complex structures connecting $J_0$ to the standard integrable complex structure on $\cptwo$. Now, under sufficiently generic conditions, given a certain number of points (the number depends on the degree $d$), there is a unique $J$-holomorphic curve of degree $d$ through those points. Therefore by fixing the appropriate number of points, and varying $J$ in the family $J_t$, one can find a family of $J_t$-holomorphic curves interpolating between the original symplectic curve and a curve which is complex in the standard sense. The problem is that various types of degenerations can occur in the curves in this family which prevent it from providing an isotopy. The analytic difficulties involved in avoiding such degenerations have prevented further progress in this direction after Siebert and Tian's 2005 paper.
	
	In this article, we provide a new approach to solving this problem. We will prove that the symplectic isotopy problem can be solved if one can find a certain collection of disjoint Lagrangian disks with boundary on the surface. While pseudoholomorphic curves provide powerful rigid techniques, one has much greater flexibility to construct isotopies of symplectic surfaces if they are not required to be pseudoholomorphic with respect to a specific almost complex structure. Our reduction utilizes this flexibility in constructing a family of surfaces, as well as Gromov's pseudholomorphic curve proof of the symplectic isotopy problem for degree one curves. Using the Lagrangian disks, we construct a certain nodal symplectic curve (a \emph{generic line arrangement}). The rigid techniques yield a symplectic isotopy of the generic symplectic line arrangement to a complex line arrangement. We use symplectic models to carry a smooth symplectic surface along the guiding nodal symplectic surfaces. We also verify that the existence of the hypothesized disjoint Lagrangian disks is a necessary as well as sufficient condition for the existence of a symplectic isotopy to a complex curve.
	
	More precisely, for a genus $g=(d-1)(d-2)/2$ surface $\Sigma$ we are interested in finding a collection of disjointly embedded Lagrangian disks $D_1,\cdots, D_k$ in $\cptwo$ ($k=d(d-1)/2$) with boundaries on $\Sigma$ satisfying the following property. 
		
		\begin{definition}
			A collection of disjoint simple closed curves $C_1,\cdots, C_k$ on a surface $\Sigma$ is a \emph{$d$-splitting system} if $\Sigma\setminus (C_1\cup\cdots \cup C_k)$ is a collection of $d$ genus zero components each with $d-1$ boundary components, where each $C_i$ separates a distinct pair of components.
		\end{definition}
		
	It follows immediately from the definition that $k=\frac{d(d-1)}{2}$. An Euler characteristic calculation implies that a $d$-splitting system can only exist on a genus $g=\frac{(d-1)(d-2)}{2}$ surface. Additionally, the set of $d$-splitting systems on a genus $\frac{(d-1)(d-2)}{2}$ surface is unique up to diffeomorphism of $\Sigma$.
	
	The main theorem is as follows.
	
	\begin{theorem}\label{thm:main}
		Suppose $\Sigma\subset \cptwo$ is a symplectically embedded surface representing the homology class $d[\cpone]\in H_2(\cptwo;\Z)$. The following two statements are equivalent:
		\begin{enumerate}
			\item There exist $k:=\frac{d(d-1)}{2}$ disjointly embedded Lagrangian disks $D_1,\cdots, D_k$ whose boundaries form a $d$-splitting system in $\Sigma$ and whose interiors are disjoint from $\Sigma$.
			\item There exists an isotopy of $\Sigma$ through symplectically embedded surfaces to a smooth complex algebraic curve.
		\end{enumerate}
	\end{theorem}
	
	Note that every 1-parameter family of symplectic surfaces can be realized as the image of the initial symplectic surface under an ambient symplectic isotopy by the following classical theorem.
	
	\begin{prop}[{\cite[Theorem 1.12]{AcDS} or \cite[Proposition 0.3]{ST}}]\label{p:ambient}
		Let $(M,\omega)$ be a symplectic manifold and $\Sigma_t\subset M$ a 1-parameter family of symplectic submanifolds. Then there exists a family of (Hamiltonian) symplectomorphisms $\Psi_t: (M,\omega)\to (M,\omega)$ such that $\Psi_0=id$ and $\Sigma_t = \Psi_t(\Sigma_0)$ for every $t$.
	\end{prop}
	
	To use Theorem \ref{thm:main} to find a solution to the symplectic isotopy problem requires finding the hypothesized disjoint Lagrangian disks with boundaries on $\Sigma$. The search for disjoint Lagrangian submanifolds can utilize a rather different set of tools in symplectic geometry than previous attempts analyzing pseudoholomorphic curves of high genus. While this \emph{embedded} Lagrangian disk problem presents its own challenges, we hope that this new approach will renew efforts to work on this long standing open problem. Additionally, because our criterion is both necessary and sufficient, it opens the possibility of discovering counterexamples if there are Floer theoretic obstructions to finding such Lagrangian disks.
	
	\begin{remark}The isotopy problem for algebraic curves of degree $d$ in $\cptwo$ is easily solved, since the moduli space of degree $d$ curves is a connected complex moduli space, and the singular curves form a subset of positive complex codimension (thus real codimension at least $2$). This is because complex curves of degree $d$ are parameterized by the coefficients of the degree $d$ monomials. The conditions that the polynomial cutting out the curve is not smooth are algebraic equations in these complex coefficients and thus the space of non-smooth curves has real co-dimension at least two. On the other hand, in the smooth category, it is not true that every smooth surface is isotopic to a complex curve, even after requiring the surface to have the same homotopy invariants as a complex curve. Finashin \cite{F} constructed examples of infinitely many smooth surfaces in $\cptwo$ whose genus and homology class match those of a complex algebraic curve and whose complements have the same fundamental groups as those for algebraic curves, but which are not smoothly isotopic (or even equivalent by a diffeomorphism of $\cptwo$) to a complex algebraic curve. These surfaces are constructed by a certain kind of surface knotting (annular rim surgery) and are distinguished by the Seiberg-Witten invariants of their branched double covers.
	\end{remark}	
	
	\subsection{Nodal case}
	
	The symplectic isotopy problem is interesting to study for curves with singularities as well. We require the singularities to be modeled on complex singularities, the simplest of which are nodal points, where two branches of the surface intersect transversally and positively. It is expected that symplectic surfaces with only nodal singularities are symplectically isotopic to complex curves. For curves of certain low degrees and small numbers of cusps and nodes the result has been proven (see \cite{OO} for the cuspidal cubic and \cite{Sh2} for irreducible low genus nodal curves). The author proved for all degrees that a symplectic isotopy to a complex curve exists in certain completely reducible cases (line arrangements) in \cite{S} (and we utilize this result in this article).
	
	The nodal version of Theorem \ref{thm:main} follows from the same proof, though it requires a stronger (but true) statement about deformations of nodal algebraic curves in $\cptwo$. Namely, that within the class of complex algebraic curves one can smooth each node independently while keeping the other nodes.
	
	\begin{theorem}\label{thm:nodal}
		Suppose $\Sigma\subset \cptwo$ is a symplectic surface which is smooth except for $n$ nodal singularities such that $\Sigma$ represents the homology class $d[\cpone]\in H_2(\cptwo;\Z)$. The following two statements are equivalent:
		\begin{enumerate}
			\item There exist $k-n=\frac{d(d-1)}{2}-n$ disjointly embedded Lagrangian disks $D_1,\cdots, D_{k-n}$ with boundaries a collection of disjoint simple closed curves in $\Sigma\setminus U$ forming a $d$-splitting system for some $U$ a union of small neighborhoods of the nodal points in $\Sigma$.
			\item There exists an isotopy of $\Sigma$ through symplectic surfaces with $n$ nodes to a nodal complex algebraic curve.
		\end{enumerate}
	\end{theorem}

	\subsection{Formulation via Lagrangian disks in symplectic fillings} \label{s:disks}
	
	A symplectic degree $d>0$ surface in $\cptwo$ has a neighborhood with concave contact type boundary. The boundary is a circle bundle over $\Sigma$ with Euler class $d^2$, and the induced contact structure is invariant with respect to the circle action obtained by rotating the fibers. This contact manifold is called the prequantization bundle of $\Sigma$, $P_{d^2}(\Sigma)$. We note that the contact structure on the prequantization bundle is compatible with an open book decomposition whose pages are surfaces of genus $g=(d-1)(d-2)/2$ with $d^2$ boundary components, and whose monodromy is a product of right handed Dehn twists about one curve parallel to each boundary component \cite{GM}. (We will focus on the smooth case, but the complement of a nodal curve also has a concave neighborhood, where the induced contact structure is supported by the same open book as for the smooth case except that the monodromy has additional left handed Dehn twists along the vanishing cycles of the nodes.)
	
	The complement of the concave neighborhood is a convex filling of this prequantization bundle. Note that the symplectic filling in the complement of a symplectic surface is symplectomorphic to the complement of a complex curve in $\cptwo$ if and only if the symplectic surface is isotopic to a complex curve. This follows from the fact that the symplectic mapping class group of $\cptwo$ deformation retracts to $PU(2,\C)$ \cite{Gr}. While its precise symplectic topology depends on the symplectic isotopy class of $\Sigma$, the homology of $\cptwo\setminus \nu(\Sigma)$ can be determined by the long exact sequence in cohomology of the pair $(\cptwo, \overline{\nu(\Sigma)})$. 
	%
	$$H_1(\cptwo\setminus \nu(\Sigma))\cong \Z/d\Z, \qquad H_2(\cptwo\setminus \nu(\Sigma))\cong \Z^{2g}, \qquad H_3(\cptwo \setminus \nu(\Sigma))=0$$
	
	\begin{example}
		In the case $d=2$, $P(\Sigma)=(L(4,1),\xi_{can})$ is the lens space $L(4,1)$ with the canonical contact structure obtained as a quotient of $(S^3,\xi_{std})$ by the $\Z/4$ action. This contact manifold has exactly two strong symplectic fillings \cite{Mc}, only one of which is a rational homology ball as is required of $\cptwo\setminus \Sigma$ by the above computations. This filling is symplectomorphic to $T^*\rptwo$. The single Lagrangian disk required by theorem \ref{thm:main} is provided by the cotangent fiber, or equivalently the co-core of the 2-handle in the Stein handle decomposition as in figure \ref{fig:rptwo}. The boundary of this curve represents $2\in \Z/4\Z = H_1(L(4,1))$.
	
	\begin{figure}
		\centering
		\includegraphics[scale=.5]{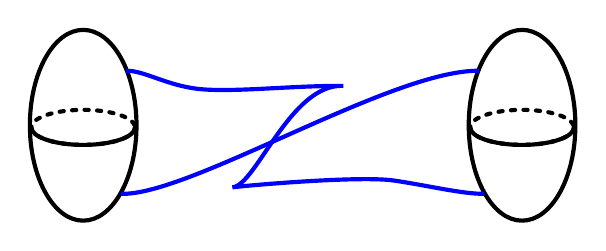}
		\caption{Stein handlebody for $T^*\rptwo=\cptwo \setminus Q$ the complement of a degree $2$ curve.}
		\label{fig:rptwo}
	\end{figure}
	\end{example}
	
	In general, we do not have classifications of the symplectic fillings of $P_{d^2}(\Sigma)$ (otherwise we would have already solved the symplectic isotopy problem). Using theorem \ref{thm:main}, we do not need to know exactly what the symplectic filling is, but rather that any filling of the fixed contact boundary with the same homology as the complement of an algebraic curve also contains that same collection of Lagrangian disks. One might hope that knowing the contact boundary and the homology of the filling would suffice to give Floer theoretic information about whether such Lagrangian disks exist. The author would be enthusiastic to build off of any Floer theoretic input to explore its geometric consequences.
	
	\textbf{Acknowledgments:} Many thanks to Vivek Shende and Ravi Vakil for helpful conversations, and to the referee for catching a gap. The author has been supported by an NSF Postdoctoral Research Fellowship Grant No. 1501728.
	
	\section{Symplectic models}
	
	To prove Theorem \ref{thm:main} we will need to patch in certain model deformations. In this section we describe these models.
	
	\subsection{From the perspective of a Lagrangian disk} \label{s:model1}
	Our first model is for a coordinate neighborhood of a Lagrangian disk with boundary on a symplectic annulus (thought of as the piece of our symplectic surface $\Sigma$ in the neighborhood of the boundary of the Lagrangian disk).
	
	In $T^*\R^2$ with coordinates $(q_1,p_1,q_2,p_2)$ and symplectic form $\omega = dq_1\wedge dp_1+dq_2\wedge dp_2$, let $D^{ml}$ be the unit disk in the $0$-section (the $(q_1,q_2)$-plane). Let $\Sigma^{ml}$ in a neighborhood of $D^{ml}$ be the image of $(\theta,s)\in S^1\times [-c,c]$ under the embedding
	$$A(\theta,s) = (\cos(\theta),-s\sin(\theta), \sin(\theta), s \cos(\theta)).$$
	$\Sigma^{ml}$ is symplectic because
	$$\omega\left( \frac{\partial A}{\partial \theta}, \frac{\partial A}{\partial s} \right)=\sin^2(\theta)+\cos^2(\theta)=1>0.$$
	$D^{ml}$ and $\Sigma^{ml}$ intersect along $\partial D^{ml}$. Since $D^{ml}$ is Lagrangian and $\Sigma^{ml}$ is symplectic, they intersect cleanly. We define the following small push-offs of $D^{ml}$:
	$$D_-^{ml} = \{(q_1,-\varepsilon q_2, q_2, \varepsilon q_1) \mid q_1^2+q_2^2\leq 1\}$$
	$$D_+^{ml} = \{(q_1, \varepsilon q_2, q_2, -\varepsilon q_1 ) \mid q_1^2+q_2^2\leq 1  \}$$
	The signs indicate the comparison of the symplectic orientation with the orientation induced by the projection to the zero section oriented $(\partial_{q_1},\partial_{q_2})$. More specifically,
	$$TD_-^{ml} = \langle \partial_{q_1}+\varepsilon \partial_{p_2}, \partial_{q_2}-\varepsilon \partial_{p_1} \rangle \qquad \qquad \omega(\partial_{q_1}+\varepsilon \partial_{p_2}, \partial_{q_2}-\varepsilon \partial_{p_1} )=-2\varepsilon<0$$
	$$TD_+^{ml} = \langle \partial_{q_1}-\varepsilon \partial_{p_2}, \partial_{q_2}+\varepsilon \partial_{p_1} \rangle \qquad \qquad \omega(\partial_{q_1}-\varepsilon \partial_{p_2}, \partial_{q_2}+\varepsilon \partial_{p_1} )=2\varepsilon>0$$
	Observe that we have clean intersections
	$$\Sigma^{ml}\cap D_-^{ml} = \{ (\cos\theta, -\varepsilon \sin\theta, \sin\theta, \varepsilon\cos\theta)  \} \cong S^1$$
	$$\Sigma^{ml}\cap D_+^{ml} = \{ (\cos\theta, \varepsilon \sin\theta, \sin\theta, -\varepsilon\cos\theta)  \} \cong S^1$$
	Also, the intersection $D_-^{ml}\cap D_+^{ml} = {(0,0,0,0)}$ is $\omega$-orthogonal: namely, by checking on the bases above, we see that $\omega(v,w)=0$ for any $v\in T_0D_+^{ml}$ and $w\in T_0D_-^{ml}$.
	
	One can explicitly check that concatenating bases for $TD_-^{ml}$ and $TD_+^{ml}$ which are positively oriented with respect to the symplectic form gives a positively oriented basis for $T^*\R^2$ (the change of coordinates from $(\partial_{q_1},\partial_{p_1},\partial_{q_2},\partial_{p_2})$ to $(\partial_{q_1}-\varepsilon\partial_{p_2},\partial_{q_2}+\varepsilon \partial_{p_1}, \partial_{q_2}-\varepsilon\partial_{p_1}, \partial_{q_1}+\varepsilon \partial_{p_2})$ has positive determinant). That the intersection number is $+1$ between these push-offs is also more generally implied by Proposition \ref{p:pushoff} which gives a more topological reason.
	
	The standard coordinates on $\R^2$ induce a trivialization of the cotangent fibration normal to the disk. Restricting this to the boundary of the disk, we get a standard trivialization (framing) of $\nu(D)|_{\partial D}\to \partial D$ which extends over $D$. We let this framing represent the $0$ framing.
	
	\begin{prop}\label{p:pushoff}
		Let $D$ be a Lagrangian disk with boundary in a symplectic surface $\Sigma$. Then the framing induced by $\Sigma$ by taking the normal direction in $\Sigma$ to $\partial D$ is $+1$.
	\end{prop} 
	
	\begin{proof}
		Identify a neighborhood of $D$ with a neighborhood of a disk in $T^*\R^2$. Because $\Sigma$ contains $\partial D$ and is symplectic, its tangent space is spanned by $T\partial D$ and another vector field $V$ along $\partial D$ which is not symplectically orthogonal to $T\partial D$. Since $\partial D$ is one dimensional, $(T\partial D)^{\perp_\omega}$ is an orientable 3-dimensional sub-bundle of $T\cptwo|_{\partial D}$, thus it is separating. The vector field $V$ must always lie on one side of this separating subspace. $(T\partial D)^{\perp_\omega}$ contains the tangent space to the $0$-section $T\R^2$, as well as the co-normal directions to $\partial D$. Therefore the component of $V$ in the normal direction to $T\R^2$ must be transverse to the co-normal line field. Since the co-normal line field to the boundary of a disk induces a framing $+1$ relative to the trivialization, (since the Gauss map of $\partial D$ has degree $1$), the normal component of $V$ must also induce a $+1$ framing of $\partial D$.
	\end{proof}
	
	A useful property of the above model is that each of the surfaces of interest intersects the $(q_1,p_2)$ plane in a 1-dimensional subspace as shown in figure \ref{fig:slice}, and the surfaces themselves are given by the orbit of these 1-dimensional subspaces under the following circle action:
	$$\theta \cdot (q_1,p_2) = (q_1\cos\theta, -p_2\sin\theta, q_1\sin\theta, p_2\cos\theta).$$
	
	\begin{figure}
		\centering
		\includegraphics[scale=1]{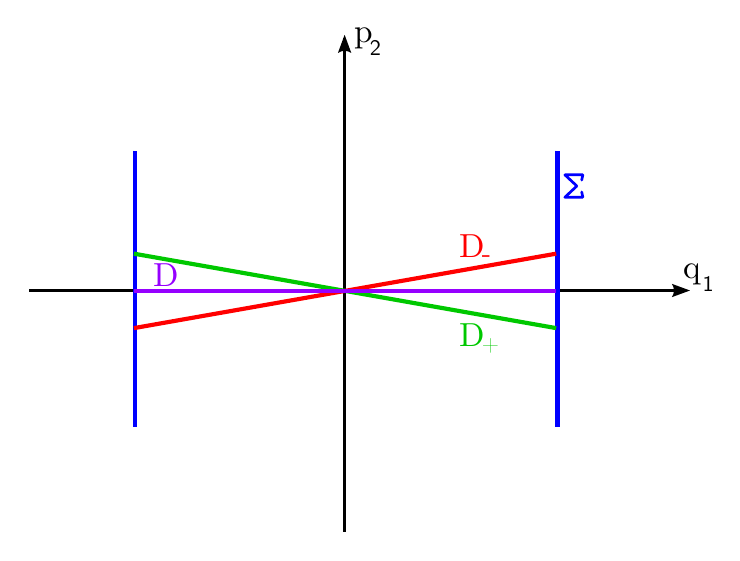}
		\caption{The $(q_1,p_2)$-plane slice of $\Sigma^{ml}$, $D^{ml}$, $D^{ml}_+$, and $D^{ml}_-$.}
		\label{fig:slice}
	\end{figure}
	
	\begin{figure}
		\centering
		\includegraphics[scale=1.4]{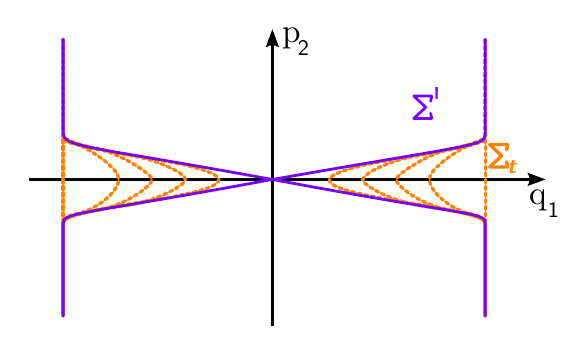}
		\caption{Curves in orange yield symplectic surfaces deforming $\Sigma^{ml}$ to become close to the nodal surface $\Sigma'$.}
		\label{fig:deform}
	\end{figure}
	
	We can construct more surfaces in this model which intersect $\Sigma^{ml}, D^{ml}, D^{ml}_+, D^{ml}_-$ in a controlled way, by drawing a 1-dimensional curve in the $(q_1,p_2)$ plane and taking the orbit under the same circle action. This idea is used for Polterovich's Lagrangian surgery \cite{Pol}, but with a different circle action (since we want symplectic instead of Lagrangian surfaces.) If we parametrize the curve as $q_1=a(s), p_2=b(s)$ then the surface obtained by applying the circle action is parametrized by
	$$C(\theta,s) = (a(s)\cos\theta,-b(s)\sin\theta, a(s)\sin\theta, b(s)\cos\theta )$$
	The symplectic form evaluates on a basis of the tangent space to this surface as
	$$\omega\left(\frac{\partial C}{\partial \theta}, \frac{\partial C}{\partial s} \right) 
	=a'(s)b(s)+b'(s)a(s).$$
	Note that $\Sigma^{ml}$ is the surface corresponding to, $a(s)=1$, $b(s)=s$. In general, the above computation shows that we get a symplectic surface whenever $a'(s)b(s)+b'(s)a(s)$ is positive. Observe that this quantity is positive for any curve such that
	$$\begin{cases}
	a(s)>0 &\\
	b'(s) >0 & \\
	a'(s) \leq 0 & \text{ whenever } b(s)\leq 0\\
	a'(s) \geq 0 & \text{ whenever } b(s)\geq 0\\
	\end{cases}$$

	For example rotating by the circle action the curves appearing in orange in figure \ref{fig:deform} gives a collection of symplectic surfaces deforming $\Sigma^{ml}$ approaching $D_-^{ml}\cup D_+^{ml}$.
	
	We also get an immersed surface $\Sigma'$ which agrees with $D_-^{ml}\cup D_+^{ml}$ away from their boundaries and agrees with $\Sigma^{ml}$ outside a neighborhood of the annulus between $\partial D_-^{ml}$ and $\partial D_+^{ml}$. To do this simply smooth the corner between the intersections of $\Sigma^{ml}$ with $D_{\pm}^{ml}$ in the $(q_1,p_2)$ plane in figure \ref{fig:slice} and take the image under the circle action.
	
	\subsection{From the perspective of a symplectic node}\label{s:model2}
	Our second model rotates the coordinates to keep track of submanifolds relative to a fixed symplectically immersed surface in a neighborhood of an $\omega$-orthogonal node modeled by
	$$\{(x_1,y_1,0,0)\mid x_1^2+y_1^2\leq \varepsilon_1^2 \}\cup \{(0,0,x_2,y_2)\mid x_2^2+y_2^2\leq \varepsilon_2^2\}$$ 
	in $(\R^4_{(x_1,y_1,x_2,y_2)},\omega_{std}=dx_1\wedge dy_1+dx_2\wedge dy_2)$.
	
	\begin{figure}
		\centering
		\includegraphics[scale=.5]{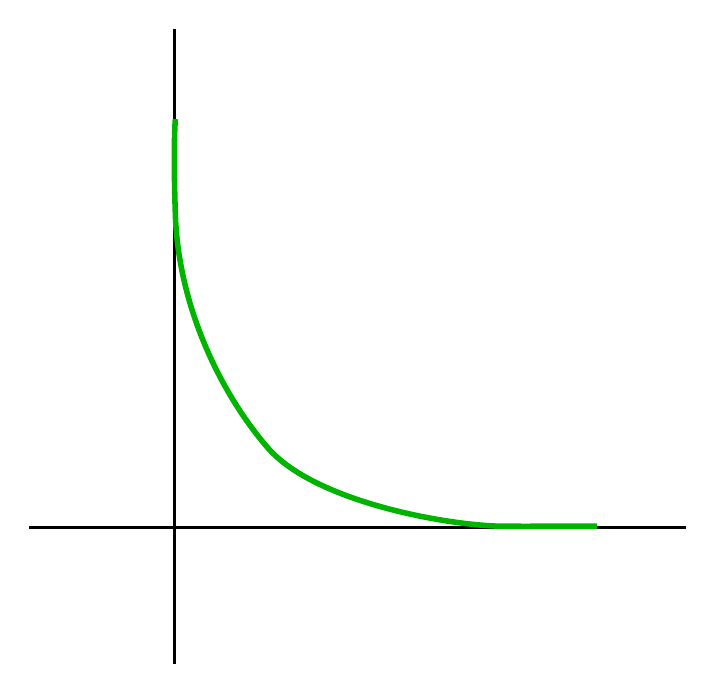}
		\caption{Curve $\gamma(s)$ in the real $(x_1,x_2)$ plane.}
		\label{fig:curve}
	\end{figure}
	
	We can sympletically smooth the node in this model, by choosing a curve $\gamma(s)=(\gamma_1(s),\gamma_2(s))$ in the real $(x_1,x_2)$-plane which follows the positive $x_1$ axis for $s$ near $0$ and follows the positive $x_2$ axis for $s$ near $1$ as in figure \ref{fig:curve}. We may arrange that $\gamma_1(s),\gamma_2(s),\gamma_2'(s)\geq 0$ and $\gamma_1'(s)\leq 0$. Also we can ensure that $\gamma_2'(s)=0$ and $\gamma_1(s)=0$ only for $s$ near $0$ and $\gamma_1'(s)=0$ and $\gamma_2(s)=0$ only for $s$ near $1$. Consider the annulus 
	$$\mathcal A(\theta,s)=\{(\gamma_1(s)e^{i\theta},\gamma_2(s)e^{-i\theta})\mid s\in I, \theta\in S^1\}.$$
	
	For $s$ near $0$, $\mathcal A(\theta,s)$ is an annular ring in the disk in the $z_1=(x_1,y_1)$-plane, and for $s$ near $1$, $\mathcal A(\theta,s)$ is an annular ring in the disk in the $z_2=(x_2,y_2)$-plane. Therefore $\mathcal A$ can be smoothly patched into the model node after deleting the parts of the disks inside the rings where $\mathcal A$ intersects. We can verify that $\mathcal A$ patches in symplectically: 
	$$\omega\left(\frac{\partial A}{\partial \theta},\frac{\partial A}{\partial s}\right)=\gamma_2(s)\gamma_2'(s)-\gamma_1(s)\gamma_1'(s)>0$$
	where the positivity comes from the assumptions on $\gamma$ above.
	
	\subsection{Rotating nodal smoothings through $\omega$-angles} \label{s:model3}
	
	Typical nodal singularities of symplectic curves need not be $\omega$-orthogonal (though they can be symplectically isotoped to be $\omega$-orthogonal see section \ref{s:orth}). Even nodal singularities of genuine complex curves need not be $\omega$-orthogonal. We will consider a neighborhood of a complex linear nodal point where the $\omega$-angles between the two branches is a varying parameter. Near the nodal curve, we provide a model symplectic smoothing which specializes to that of section \ref{s:model2} in the $\omega$-orthogonal case.
	
	Consider a pair of distinct complex lines intersecting at the origin in $(\R^4_{(x_1,y_1,x_2,y_2)}, dx_1\wedge dy_1+dx_2\wedge dy_2)$ where we set $z_j=x_j+iy_j$. We can sympletically rotate one of the lines to $\{z_2=0\}$ so the pair is given by $\{(z_1-(a+bi)z_2)z_2=0\}$. A complex deformation of this nodal singularity to a smoothing is given by $\{(z_1-(a+bi)z_2)z_2=\varepsilon\}$. We can parametrize this deformation by choosing $\gamma_1(s)$ and $\gamma_2(s)$ satisfying $\gamma_1(s)\gamma_2(s)=\varepsilon$ and parametrizing by
	$$z_2= \gamma_2(s)e^{-i\theta} \qquad z_1=(a+bi)\gamma_2(s)e^{-i\theta}+\gamma_1(s)e^{i\theta}$$
	\begin{align*}
	\A_{a+bi}(\theta,s) =& \left((\gamma_1(s)+a\gamma_2(s))\cos\theta+b\gamma_2(s)\sin\theta , (\gamma_1(s)-a\gamma_2(s))\sin\theta+b\gamma_2(s)\cos\theta,\right. \\
	&  \left. \gamma_2(s)\cos\theta, -\gamma_2(s)\sin\theta \right).
	\end{align*}
	
	We can slightly modify the curve to $\widetilde\gamma(s)=(\widetilde\gamma_1(s),\widetilde\gamma_2(s))$ to agree with $\{(t,0)\}\cup\{(0,u)\}$ outside a neighborhood of the origin. This has the effect of making the image of the corresponding parametrization $\widetilde{\A}_{a+bi}$ agree with the lines $(z_1-(a+bi)z_2)z_2=0$ outside a neighborhood of the origin.  If we choose $\varepsilon$ sufficiently small, we can make this modification $C^1$ small so that the image of the resulting smoothing $\A_{a+bi}$ is $C^1$ close to the complex curve $\{(z_1-(a+bi)z_2)z_2=\varepsilon\}$ and thus is symplectic and agrees with the complex lines $\{(z_1-(a+bi)z_2)z_2=0\}$ outside a chosen neighborhood of the origin.

	\section{Nodal curves} \label{s:line}
	
	If $\Sigma'$ is the union of symplectically embedded spheres in $\cptwo$ with each component homologous to $\cpone$ such that each pair of components intersects exactly once (thus positively), we say that $\Sigma'$ is a \emph{symplectic line arrangement}. If each distinct pair of components intersects at a distinct point, we say that the symplectic line arrangement is generic. The following is a special case of \cite[Proposition 4.1]{S}.
	
	\begin{prop}[\cite{S}]\label{p:lineiso}
		Every generic symplectic line arrangement is isotopic through generic symplectic line arrangements to a complex line arrangement.
	\end{prop}
	
	The rough outline of the proof is to use Gromov's result that for each almost complex structure $J$ compatible with the standard symplectic structure on $\cptwo$, any two distinct points contain a unique $J$-holomorphic ``line''. Furthermore, any two compatible almost complex structures are connected by a path of compatible almost complex structures. Therefore, starting with an almost complex structure $J$ for which $\Sigma'$ is $J$-holomorphic, there is a path $J_t$ from $J$ to the standard integrable complex structure $J_{std}$. Choosing two points on each component of $\Sigma'$ and varying through the family $J_t$ gives a family of $J_t$ holomorphic lines interpolating between $\Sigma'$ and a complex line arrangement. The only problem that could occur in this family is that the lines may pass through non-generic intersections (intersections between more than two components may coincide at certain times). We show in \cite{S} that a perturbation of the choice of points ensures that we avoid this non-generic behavior.
	
	Complex projective curves in $\cptwo$ of degree $d$ are parametrized by the coefficients of a homogeneous degree $d$ polynomial. Some of these curves have singularities, but the space of coefficients corresponding to singular curves has positive complex co-dimension. Curves with only nodal singularities correspond to a complex co-dimension one normal crossing divisor. Therefore each nodal singularity has a one dimensional complex family of smoothings. Moreover these nodal singularities can be smoothed independently (this normal crossing property is special to algebraic curves in $\cptwo$). See \cite{Ful} for a reference on these facts about the space of nodal algebraic plane curves. In a neighborhood of a nodal point on a curve, there are complex algebraic coordinates $(Z_1,Z_2)$ such that the algebraic curves obtained by varying the normal parameter $t_i$ which smooths the node of interest is given in these coordinates by $C_{t_i}=\{Z_1Z_2 = t_i\}$, so the two sheets of the nodal curve in these coordinates are $\{Z_1=0\}$ and $\{Z_2=0\}$. Note that these complex analytic coordinates need not be compatible with Darboux coordinates in the standard way, in particular the two branches need not be $\omega$-orthogonal.
	
	\subsection{Symplectically orthogonal nodes} \label{s:orth}
	
	Generic transverse intersections of symplectic surfaces can be either positive or negative (whether the symplectic orientations of the two components add up to the 4-dimensional orientation). In this article, we restrict only to positive nodes but even these have geometric variations. The most canonical model for a positive transverse intersection of symplectic surfaces is when the two branches meet $\omega$-orthogonally as in the model of section \ref{s:model2}. More generally, we can consider two transverse symplectic planes in $(\R^4_{(x_1,y_1,x_2,y_2)},\omega=dx_1\wedge dy_1+dx_2\wedge dy_2)$. By a symplectic change of coordinates, we may assume the first plane is spanned by $(\partial_{x_1},\partial_{y_1})$, and then the positivity and transversality assumptions imply that the second plane has an oriented basis of the form 
	$$(u:=\partial_{x_2}+a\partial_{x_1}+b\partial_{y_1}, v:=\partial_{y_2}+c\partial_{x_1}+d\partial_{y_1}).$$
	This second plane is symplectic if and only if $1+ad-bc>0$. It is $\omega$-orthogonal the first plane if and only if $a=b=c=d=0$. It is complex with respect to the standard complex structure $J_{std}=i$ if and only if $d=a$ and $c=-b$.
	
	The need to find a symplectic isotopy taking positively intersecting symplectic submanifolds to $\omega$-orthogonally intersecting symplectic submanifolds arose in Gompf's symplectic sum paper \cite[Lemma 2.3]{G} and was proven there in greater generality than the 4-dimensional case we consider here. Symington made more explicit the 4-dimensional case of Gompf's construction in her thesis \cite[Lemma 3.2.3]{SyTh}, and we present a modification of that here, in order to ensure the following additional property we will need later on. We state the lemma 1-parametrically and include the parameter in the construction in the proof since this is what we require later on.
	
	\begin{lemma}\label{l:orth}
		Suppose a smoothly varying 1-parameter family of symplectic surfaces $\Sigma_t$ in $\cptwo$ has a positive transverse node at $p_t$. Then there is a symplectic isotopy $\Sigma_{t,s}$, with $\Sigma_t=\Sigma_{t,0}$, supported in an arbitrarily small neighborhood of $p_t$, such that the symplectic node of $\Sigma_{t,1}$ at $p_t$ is $\omega$-orthogonal. Moreover, if $\Sigma_1$ agrees with two intersecting complex lines in $\cptwo$ in a neighborhood of $p_1$, then $\Sigma_{1,s}$ can be chosen to agree with a family of intersecting complex lines in $\cptwo$ in a neighborhood of $p_1$ for all $s\in [0,1]$.
	\end{lemma}
	
	The second statement is not true of Symington's construction, and is more easily verified in the following simplified version of Gompf's more general construction.
	
	\begin{proof}
		Choose a smoothly varying 1-parameter family of Darboux coordinates $(x_1^t,y_1^t,x_2^t,y_2^t)$ giving a symplectomorphism of a neighborhood of $p_t$ to $U\subset (\R^4,dx_1\wedge dy_1+dx_2\wedge dy_2)$ satisfying the following properties:
		\begin{enumerate}
			\item $p_t=(0,0,0,0)$.
			\item One branch of $\Sigma_t$ agrees with $\{x_2^t=y_2^t=0\}$.
			\item In the case that $\Sigma_1$ agrees with complex lines near $p_1$, require the Darboux coordinates $(x_1^1,y_1^1,x_2^1,y_2^1)$ to identify complex lines in $\cptwo$ through $p_1$ with complex lines in $\C^2$ through the origin.
		\end{enumerate}
		
		We can find coordinates $(x_1^1,y_1^1,x_2^1,y_2^1)$ satisfying these three properties for $t=1$ when $\Sigma_1$ agrees with complex lines by first choosing affine Kahler coordinates centered at $p_1$ such that one line is identified with $z_2=0$ (using the fact that $U(2)$ is transitive on lines) and then applying Lemma \ref{l:Darboux} to get Darboux coordinates sending the complex lines to complex lines. Then we can extend these Darboux coordinates smoothly to a family of Darboux charts, satisfying the first two properties for all $t$ using the relative Moser theorem.
		
		Then the tangent space to the other branch is spanned by $(u^t:=\partial_{x_2^t}+a^t\partial_{x_1^t}+b^t\partial_{y_1^t}, v^t:=\partial_{y_2^t}+c^t\partial_{x_1^t}+d^t\partial_{y_1^t})$. Using the fact that the non-degeneracy condition is open, we can perform a $C^1$ small isotopy of $\Sigma^t$ so that the second branch agrees with the linear subspace spanned by $(u^t,v^t)$ in a $\delta$-neighborhood of $p_t$. When $\Sigma^1$ is complex linear near $p_1$ in these affine coordinates, the isotopy is trivial.
		
		Next, we will rotate the linear space $\langle u^t,v^t \rangle$ to $\langle \partial_{x_2^t},\partial_{y_2^t}\rangle$ with an appropriate cut-off function in the radial direction. By the symplectic condition $a^td^t-b^tc^t>-1$ for all $t\in [0,1]$, so there exists $\eta>0$ such that $1+(1+\eta)(a^td^t-b^tc^t)>0$ for all $t\in [0,1]$. As in \cite{G}, choose a smooth cut-off function $\mu:[0,\infty)\to [0,1]$ such that $\mu(r)=1$ for $0\leq r\leq \varepsilon \ll \delta$, $\mu(r)=0$ for $r\geq \delta$, $\mu'(r)\leq 0$ for all $r$ and satisfying the additional property that $\mu(r)+r\mu'(r)\geq -\eta$.
		
		Now for $s\in[0,1]$ let 
		$$u_{s,r}^t= \partial_{x_2^t}+(1-s\mu(r))a^t\partial_{x_1^t}+(1-s\mu(r))b^t\partial_{y_1^t}$$
		$$v_{s,r}^t= \partial_{x_2^t}+(1-s\mu(r))c^t\partial_{x_1^t}+(1-s\mu(r))d^t\partial_{y_1^t}$$
		and let $\phi_s^t(r,\theta) = r\cos\theta \cdot u_{s,r}^t+r\sin\theta \cdot v_{s,r}^t$.
		Then the image of $\phi_s^t$ is symplectic because
		$$\omega\left(\frac{\partial \phi_s^t}{\partial r}, \frac{\partial \phi_s^t}{\partial \theta}  \right)= r\left[(1-s\mu(r))(1-s(\mu(r)+r\mu'(r)))(a^td^t-b^tc^t)+1\right]>0$$
		where positivity follows from the assumption $\mu(r)+r\mu'(r)\geq -\eta$ so $$(1-s\mu(r))(1-s(\mu(r)+r\mu'(r)))(a^td^t-b^tc^t)+1>0.$$
		For $0\leq r \leq \varepsilon$, $\mu(r)$ is identically $1$ so $\phi_s^t(\{r\leq \varepsilon\})$ is a disk in the linear plane spanned by $u_s^t=\partial_{x_2^t}+(1-s)a^t\partial_{x_1^t}+(1-s)b^t\partial_{y_1^t}$ and $v_s^t=\partial_{y_2^t}+(1-s)c^t\partial_{x_1^t}+(1-s)d^t\partial_{y_1^t}$. Since $u_1^t=\partial_{x_2^t}$ and $v_1^t=\partial_{y_2^t}$, the node of $\Sigma_{t,1}$ is $\omega$-orthogonal. 
		
		Observe that if $\Sigma^1$ agrees with a pair of complex lines near $p_1$ in Kahler and thus Darboux coordinates then $a^1=d^1$ and $b^1=-c^1$, so $a_s^1:=(1-s)a^1=(1-s)d^1:=d_s^1$ and $b_s^1:=(1-s)b^1=-(1-s)c^1=-c_s^1$, so $\Sigma_{1,s}$ also agrees with a pair of complex lines through $p_1$ (in both Darboux and Kahler coordinates) for all $s$ in a sufficiently small neighborhood of $p_1$.
		
	\end{proof}
	
	\begin{lemma}\label{l:Darboux} Let $\omega_K = -dd^{\C}(\log(1+|z|^2))$ be the Kahler form on an affine chart of $\cptwo$, and let $\omega_D$ be the standard (Darboux) symplectic form on $\C^2$. There is a diffeomorphism $\Psi:\C^2 \to \C^2$ such that $\Psi^*\omega_K = \omega_D$ which preserves all complex lines through the origin.
	\end{lemma}
	
	\begin{proof}
		The Kahler form on $\cptwo$ in an affine chart with coordinates $z=(z_1,z_2)=(x_1,y_1,x_2,y_2)$ is given by $\omega_K=-dd^{\C}(\log(1+|z|^2))=2i\partial\overline{\partial}(\log(1+|z|^2))$. The standard symplectic form on $\C^2$ in Darboux coordinates $z=(z_1,z_2)=(x_1,y_1,x_2,y_2)$ is $\omega_D=-dd^{\C}(|z|^2)$.
		
		Because $[\omega_K]=[\omega_D]=0$, Moser's trick allows one to find a diffeomorphism such that $\Psi^*\omega_K=\omega_D$. Because we are interested in how this symplectomorphism acts on complex lines through the origin, we will look more closely at how $\Psi$ is generated in this case through Moser's argument.
		
		Let $\omega_t = t\omega_K +(1-t)\omega_D = -dd^{\C}(t\log(1+|z|^2)+(1-t)|z|^2)$. Then 
		$$\frac{d\omega_t}{dt} = \omega_K-\omega_D = -dd^{\C}(\log(1+|z|^2)-|z|^2).$$ 
		Let $\mu = -d^{\C}(\log(1+|z|^2)-|z|^2)$ so $d\mu = \frac{d\omega_t}{dt}$. The Moser trick is to find a vector field $V_t$ such that $\iota_{V_t}\omega_t = -\mu$, then integrating the flow of $V_t$ gives the symplectomorphism $\Psi_t$ such that $\Psi_t^*\omega_t = \omega_0=\omega_D$. We will calculate $V_t$ in this case. Let $f(z)=|z|^2$, and let $\lambda = -d^{\C}(f)$. Note that $\omega_D = d\lambda$.
		
		\begin{eqnarray*}
			\omega_t &=& -dd^{\C}(t\log(1+|z|^2)-|z|^2)\\
			& =&  d\left( \left( \frac{t}{1+|z|^2} +(1-t)  \right)\lambda \right) \\
			&=& \left( \frac{t}{1+|z|^2} +(1-t)  \right) d\lambda-\left(\frac{t}{(1+|z|^2)^2} \right)df\wedge \lambda
		\end{eqnarray*}
		
		and
		
		$$-\mu = d^{\C}(\log(1+|z|^2)-|z|^2) = -\left( \frac{1}{1+|z|^2}-1\right)\lambda$$
		
		Let $V_0 = x_1\partial_{x_1}+y_1\partial_{y_1}+x_2 \partial_{x_2}+y_2\partial_{y_2}$ be the radial vector field. Observe that $df(V_0) = |z|^2$, $\lambda(V_0)=0$, and $\iota_{V_0}d\lambda = \lambda$. Let $V_t = f_tV_0$ where $f_t$ is a function. Then
		$$\iota_{V_t}\omega_t = f_t \iota_{V_0}\omega_t =f_t\left[ \left( \frac{t}{1+|z|^2} +(1-t)  \right) -\left(\frac{t}{(1+|z|^2)^2} \right)|z|^2  \right]\lambda$$
		
		Therefore $\iota_{V_t}\omega_t = -\mu$ when
		$$f_t(z) = \frac{-\left( \frac{1}{1+|z|^2}-1\right)}{\left( \frac{t}{1+|z|^2} +(1-t)  \right) -\left(\frac{t}{(1+|z|^2)^2} \right)|z|^2 } = \frac{|z|^2(1+|z|^2)}{t+(1-t)(1+|z|^2)^2}$$
		
		In particular, $V_t$ is a scalar multiple of the radial vector field $V_0$. Therefore the diffeomorphism generated by the flow of $V_t$ preserves each real line through the origin. Therefore the Darboux chart $\phi_1$ preserves all linear subspaces through the origin, and in particular preserves the complex lines through the origin.
		
	\end{proof}
	
	\section{Building a symplectic isotopy from Lagrangian disks} \label{s:deg}
	
	Now we construct a symplectic isotopy of $\Sigma$ to a complex algebraic curve, from the assumption we have Lagrangian disks as in condition (1), giving one direction of Theorem \ref{thm:main}.
	
	\begin{proof}
	A smooth Lagrangian submanifold has a standard neighborhood modeled on its co-tangent bundle. For a Lagrangian with boundary, we obtain the same result by extending the Lagrangian submanifold to an open Lagrangian containing the closed Lagrangian with boundary. In particular, a neighborhood of a Lagrangian disk has coordinates identifying it with a subset of $T^*\R^2$, where the Lagrangian disk is identified with a disk in the $0$-section. For each of the $k$ disjoint Lagrangian disks in $\cptwo$ whose boundaries form a $d$-splitting system of $\Sigma$, we can identify the neighborhood and the submanifolds $D$ and $\Sigma$ with the model of section \ref{s:model1} after a symplectic isotopy of $\Sigma$ near $\partial D$ as follows. 
	
	First, symplectically identify a neighborhood of $D$ with a neighborhood of the unit disk in $\R^2$ inside $T^*\R^2$. We continue to write $\Sigma$ to denote its image in the subset of $T^*\R^2$ under this identification. As a sub-bundle of $T(T^*\R^2)|_{\partial D}$, $T\Sigma|_{\partial D}$ and $T\Sigma^{ml}|_{\partial D}$ are both transverse to $(T\partial D)^{\perp_\omega}$ (which is a necessary and sufficient condition for a $2$-dimensional sub-bundle containing $T\partial D$ to be a symplectic sub-bundle). Since $(T\partial D)^{\perp_\omega}$ is 3-dimensional, $T\Sigma|_{\partial D}$ and $T\Sigma^{ml}|_{\partial D}$ must be isotopic through symplectic sub-bundles of $T(T^*\R^2)|_{\partial D}$. Therefore $\Sigma$ and $\Sigma^{ml}$ induce the same framing on $\partial D$ and there is a smooth isotopy $\Sigma^t$  such that $\Sigma=\Sigma^0$, $\Sigma^t=\Sigma$ outside a small neighborhood of $\partial D$, $\Sigma^1$ agrees with $\Sigma^{ml}$ in a smaller neighborhood of $\partial D$, and $T\Sigma^t|_{\partial D}$ is a symplectic sub-bundle. The non-degeneracy condition is open, and the isotopy is symplectic along $\partial D$. Therefore, by choosing an isotopy supported in a sufficiently small neighborhood of $\partial D$, we can ensure that $\Sigma^t$ is symplectic for all $t$.
	 
	 Using the model of section \ref{s:model1}, we find an immersed symplectic surface $\Sigma'$ with a positive $\omega$-orthogonal self-intersection at a point on the Lagrangian $D$, such that $\Sigma'$ agrees with $\Sigma$ outside a neighborhood of $D$, and agrees with the model symplectic push-offs $D_\pm$ of $D$ near the self-intersection point. Moreover, we can use the model to find a symplectic isotopy of $\Sigma$ to a smooth symplectic surface which agrees with $\Sigma'$ outside of an arbitrarily small neighborhood of the node.
	 
	 Apply this construction along each of the $k$ Lagrangian disks in disjoint neighborhoods. Rename the resulting nodal symplectic surface $\Sigma'$. Then we have shown that	there is a symplectic isotopy $\Sigma_t$ of $\Sigma$ to a symplectic surface $\Sigma_{1/2}$ such that such that if $U$ is the neighborhood of the $k$ Lagrangian disks, and $U'$ is an arbitrarily small neighborhood of the nodal points of $\Sigma'$, then 
	 	\begin{itemize}
	 		\item $\Sigma_t$ agrees with $\Sigma$ outside $U$ for $t\in [0,1/2]$,
	 		\item $\Sigma_{1/2}$ agrees with the nodal symplectic curve $\Sigma'$ outside $U'$,
	 		\item $\Sigma_{1/2}$ is an arbitrarily small symplectic smoothing of $\Sigma'$ in $U'$ as in the model of section \ref{s:model1}.
	 	\end{itemize}
	 
	 Because of the $d$-splitting condition on the boundaries of the Lagrangian disks, we can verify the nodal surface $\Sigma'$ is a collection of $d$ embedded symplectic spheres such that each pair intersects once positively as follows. The construction exhibits that $\Sigma$ and $\Sigma'$ are homologous (take the image under the $S^1$ action of the 2-dimensional region in the $(q_1,p_2)$ plane bounded by the curves). Therefore, the total homology class of the components of $\Sigma'$ is $d[\cpone]$. Because the components of $\Sigma'$ intersect once positively, they each represent non-trivial homology classes. Because they are symplectic spheres they must have degree $1$ or $2$. Since there are $d$ components and the sum of the homology classes is $d[\cpone]$, it follows that each sphere represents the class $[\cpone]$ so $\Sigma'$ is a symplectic line arrangement. Set $\Sigma_{1/2}':=\Sigma'$. Now use Proposition \ref{p:lineiso} to find a symplectic isotopy $\Sigma_t'$ for $t\in [1/2,3/4]$ to a complex line arrangement $\Sigma'_{3/4}$ (with respect to the standard complex structure on $\cptwo$).
	
	For the family of generic symplectic line arrangements $\Sigma'_t$, $t\in [\frac{1}{2},\frac{3}{4}]$, apply Lemma \ref{l:orth} to get a symplectic isotopy $\Sigma_{t,s}'$ supported in neighborhoods of the nodal points such that $\Sigma_{t,0}'=\Sigma_t'$ and the nodes of $\Sigma_{t,1}'$ are $\omega$-orthogonal. Let $\widetilde \Sigma_t:=\Sigma_{t,1}'$ denote the resulting family of $\omega$-orthogonal symplectic line arrangements. Since the nodes of $\Sigma_{1/2}'$ are $\omega$-orthogonal by the model in section \ref{s:model1}, $\Sigma_{1/2}'=\widetilde\Sigma_{1/2}$.
	
	Then by the relative Moser theorem, there exist sufficiently small neighborhoods of the $\omega$-orthogonal nodal points of $\widetilde \Sigma_t$, where we can find a smooth family of Darboux coordinates centered at the nodal points such that $\widetilde \Sigma_t$ is given in this coordinate chart by $\{(x_1,y_1,0,0)\}\cup \{(0,0,x_2,y_2)\}$ as in the model of section \ref{s:model2}. Moreover, we can choose this Darboux chart so that a smoothing of the model of section \ref{s:model1} coincides with the smoothing of the model in section \ref{s:model2}.
	
	More specifically, the symplectomorphism defined by
	$$x_1 = \varepsilon q_1+p_2, \qquad y_1 = p_1-\varepsilon q_2, \qquad x_2 = p_2-\varepsilon q_1, \qquad y_2 = -p_1 - \varepsilon q_2$$
	takes the disk $D^{ml}_-$ to $\{(2\varepsilon q_1, -2\varepsilon q_2, 0, 0)\mid q_1^2+q_2^2=1\}$ and $D^{ml}_+$ to $\{(0,0, -2\varepsilon q_1, -2\varepsilon q_2) \mid q_1^2+q_2^2=1\}$. It takes the smoothing annulus
	$(a(s)\cos\theta, b(s)\sin\theta, -a(s)\sin\theta, b(s)\cos\theta)$ to 
	$$\Big((\varepsilon a(s)+b(s))\cos\theta,(\varepsilon a(s)+b(s))\sin\theta, (-\varepsilon a(s)+b(s))\cos(-\theta), (-\varepsilon a(s)+b(s))\sin(-\theta) \Big)$$
	$$=(\gamma_1(s)e^{i\theta}, \gamma_2(s)e^{-i\theta})$$
	when $\gamma_1(s)= \varepsilon a(s)+b(s)$ and $\gamma_2 = -\varepsilon a(s)+b(s)$.
	
	Then implanting the parametric family of Darboux charts into the neighborhoods of the nodes over the isotopy $\widetilde{\Sigma}_t$, yields a symplectic isotopy from $\Sigma_{1/2}$ through smoothings $\Sigma_t$ of the nodes of $\widetilde{\Sigma}_t$ for $t\in [1/2,3/4]$.
	
	Since $\Sigma_{3/4}'$ is a genuine complex line arrangement, it has a complex deformation smoothing its nodes to a smooth complex curve $\Sigma_1$. Finally, we connect $\Sigma_{3/4}$ to $\Sigma_1$ through a symplectic isotopy as follows. $\Sigma_{3/4}$ agrees with $\widetilde{\Sigma}_{3/4}=\Sigma_{3/4,1}'$ outside a small neighborhood $V$ of the nodes. We can assume by Lemma \ref{l:orth} that inside each component of $V$, ${\Sigma}_{3/4,s}'$ agrees with some pair of complex lines for each $s$. In $V$, we can interpolate $\Sigma_{3/4}$ to a symplectic smoothing of $\Sigma_{3/4,0}'$ through symplectic smoothings of $\Sigma_{3/4,s}'$ using the model in section \ref{s:model3} in the affine Darboux chart utilized for $\Sigma'_{3/4,s}$ in Lemma \ref{l:orth}. We can extend this isotopy outside $V$ to agree with $\Sigma_{3/4,s}'$. Because we use the affine chart where $\Sigma_{3/4,0}'$ is visibly linear, the smoothing of $\Sigma_{3/4,0}'$ by section \ref{s:model3} is $C^1$ close to a complex deformation agreeing to first order with $(z_1-(a+bi)z_2)z_2=\varepsilon$. Therefore there is a symplectic isotopy from this smoothing of $\Sigma_{3/4,0}'$ to the complex curve $\Sigma_1$.
	
	We summarize the collection of isotopies constructed in the following diagram where the smooth symplectic isotopy runs around the upper right boundary.
	
	$$\xymatrixcolsep{4pc} \xymatrix{
		{\color{blue}\Sigma=\Sigma_0} \ar[rdd]_{\S 2.1 \text{ nodalize}} \ar[r]^{\S 2.1 \text{ smooth}}& {\color{blue}\Sigma_{1/2}} \ar@{-}[r]  & \ar[r] & {\color{blue}  \Sigma_{3/4}} \ar@{=}[rd] & \\
														& \widetilde{\Sigma}_{1/2}\ar@{-}[r] \ar[u]  &- \ar@{=>}[u]_{\S 2.2 \text{ smooth}} \ar[r] &  \widetilde{\Sigma}_{3/4} \ar[u] \ar[r]_{\S 2.3 \text{ smooth}} & {\color{blue}\widetilde{\Sigma}_1}\\
														& \Sigma'=\Sigma_{1/2}'  \ar@{=}[u] \ar@{-}[r]_>{\text{Proposition } 3.1}  & - \ar@{=>}[u]_{\text{Lemma } 3.2} \ar[r]&  \Sigma_{3/4}' \ar[u] \ar[dr]_{\text{complex smooth}} \ar[r]^{\S 2.3 \text{ smooth}} & {\color{blue}\Sigma_1'} \ar[u] \ar[d]^{C^1 \text{ close}}\\
														&						& 		&	& {\color{blue}\Sigma_1}\\
		}$$

	\end{proof}
	
	The nodal case follows precisely the same steps, we just require fewer Lagrangian disks to degenerate along to obtain a symplectic line arrangement. Since the symplectic degenerations are performed locally near the Lagrangian disks which are disjoint from $\Sigma$ except along their boundaries, we simply keep some of the nodes of the symplectic line arrangement $\Sigma_t'$ instead of resolving them. When we deform the complex line arrangement to a curve with some but not all nodes, we require the stronger result about deformations of complex planar nodal curves which allows us to deform the algebraic curve by independent complex parameters $t_i$ where the neighborhoods of the $i^{th}$ node has complex analytic coordinates where the curves are given by $\{ Z_1Z_2=t_i\}$.

	\section{The converse} \label{s:converse}
	Finally, we prove the reverse direction of Theorem \ref{thm:main}.
	
	\begin{proof}
	Suppose there is a family of smooth symplectic surfaces $\Sigma_t$ for $t\in [0,1]$ where $\Sigma_1$ is a complex curve of degree $d$. The space of all degree $d$ algebraic curves is connected. The curves with nodal singularities can be deformed to smooth curves as discussed in section \ref{s:line}. Therefore there is a path $\Sigma_t$ for $t\in [1,2]$ of smooth complex algebraic curves from $\Sigma_1$ to a complex line arrangement $\Sigma_2'$. 
	
	We have a smooth symplectic isotopy from $\Sigma_0$ to $\Sigma_{2-\varepsilon}$ for $\varepsilon>0$. We will show that for sufficiently small $\varepsilon$, $\Sigma_{2-\varepsilon}$ is symplectically isotopic to a certain nearby smooth symplectic surface $\Sigma_3$ such that there are embedded Lagrangian disks with boundary on $\Sigma_{3}$ in a neighborhood of each node of $\Sigma_2'$. Because $\Sigma_2'$ is a line arrangement, the boundaries of these disks will form a $d$-splitting system of $\Sigma_{3}$. To isotope $\Sigma_{2-\varepsilon}$ to such a $\Sigma_3$, first make a $C^1$ small isotopy of $\Sigma_{2-\varepsilon}$ to a smooth symplectic surface that agrees with $\Sigma_2$ outside a small neighborhood of the nodes and agrees with the model of section \ref{s:model3} in the neighborhood of the nodes. Then follow an $\omega$-orthogonalizing isotopy given by Lemma \ref{l:orth} using a parametrized family of smoothing models of section \ref{s:model3}. At the end of this isotopy the nodal surface $\Sigma_3'$ is $\omega$-orthogonal, and the smooth symplectic surface $\Sigma_3$ agrees with the model smoothing of section \ref{s:model2} obtained by rotating a curve $\gamma(s)=(\gamma_1(s),\gamma_2(s))\in \R^2_{(x_1,x_2)}$ by the given circle action. We can arrange that $\gamma$ passes through $(\varepsilon',\varepsilon')$ for some $\varepsilon'>0$. Then a Lagrangian disk in this model with boundary on the symplectic surface $\Sigma_3$ is given by $\{(r\cos\theta, r\sin\theta, r\cos\theta, -r\sin\theta)\mid r\leq \varepsilon' \}$. Because these disks are in arbitrarily small neighborhoods of the nodes, they are disjointly embedded and intersect $\Sigma_3$ only along their boundaries.
		
	Finally, the smooth symplectic isotopy from $\Sigma_0$ to $\Sigma_3$ is realized by an ambient symplectic isotopy by Proposition \ref{p:ambient} giving a family of symplectomorphisms $\Psi_t$ for $t\in [0,3]$ such that $\Psi_0=id$ and $\Sigma_t=\Psi_t(\Sigma_0)$. Since Lagrangians are sent to Lagrangians under symplectomorphisms, the converse direction in Theorem \ref{thm:main} then follows by pulling back the Lagrangian disks with boundary on $\Sigma_{3}$ by the symplectomorphism $\Psi_{3}$.
	\end{proof}

	\bibliography{references}
	\bibliographystyle{alpha}
	
\end{document}